\documentclass[10pt,a4paper]{article}

\usepackage{amsmath,amsthm,amsopn,amsfonts,amssymb,dsfont,color}
\usepackage[dvips]{graphicx}
\usepackage{psfrag}

\newtheorem{lem}{Lemma}[section]

\newtheorem{thm}[lem]{Theorem}
\newtheorem{assumption}[lem]{Assumption}
\newtheorem{prop}[lem]{Proposition}
\theoremstyle{definition}

\theoremstyle{remark}
\newtheorem{rem}[lem]{Remark}

\numberwithin{equation}{section}

\newcommand{\ep}{\varepsilon}
\newcommand{\ue}{u^\ep}

\newcommand{\R}{\mathbb{R}}

\newcommand{\slope}{\widehat \alpha}
\newcommand{\EB}{e^{-\beta t/\ep }}

\title{Propagating interface in a monostable reaction-diffusion equation with time delay}
\date{}

\begin{document}
\maketitle \vspace{-15 mm}
\begin{center}
{\large\bf Matthieu Alfaro }\\[1ex]
I3M, Universit\'e de Montpellier 2,\\
CC051, Place Eug\`ene Bataillon, 34095 Montpellier Cedex 5, France,\\[2ex]
{\large\bf Arnaud Ducrot }\\[1ex]
IMB UMR CNRS 5251, Universit\'e de Bordeaux \\
3 ter, Place de la Victoire, 33000 Bordeaux France. \\[2ex]
\end{center}
\vspace{15pt}


\begin{abstract}

We consider a monostable time-delayed reaction-diffusion equation
arising from population dynamics models. We let a small parameter
tend to zero and investigate the behavior of the solutions. We
construct accurate lower  barriers --- by using a non standard
bistable approximation of the monostable problem--- and upper
barriers. As a consequence, we prove the convergence to a
propagating interface.\\

\footnote{

The authors are supported by the French Agence Nationale de la
Recherche within the project IDEE (ANR-2010-0112-01).}
\end{abstract}

\section{Introduction}\label{s:intro}

In this work we investigate the singular limit, as $\ep \to 0^+$,
of $\ue:[-\ep \tau,\infty) \times \R ^N\to \R$ the solution of the
delayed reaction-diffusion equation
\begin{equation}\label{eq}
\partial_t u(t,x)=\ep \Delta u(t,x) +\frac 1 \ep
\left[f(u(t-\ep \tau,x))- u(t,x)\right], \;\; t>0,\;x\in\R^N,
\end{equation}
supplemented with the initial data of delayed type
\begin{equation}\label{eq-ID}
u\left(\theta,x\right)=\varphi\left(\frac{\theta}{\ep},x\right),\;\;
-\ep\tau \leq\theta \leq 0,\; x\in \R^ N.
\end{equation}
Here $\tau>0$ is a given delay parameter; $f:[0,\infty)\to
[0,\infty)$ is a given increasing and monostable nonlinearity
--- see \eqref{function-f} for precise assumptions; the initial data
$\varphi:[-\tau,0]\times \R ^N \to \R$ is a given smooth function
--- see Assumption \ref{ASS-initial}.

Equation \eqref{eq} is widely used in population dynamics models.
In this context, $u(t,x)$ denotes the density of individuals at
time $t$ and spatial location $x$. The function $f$ is the birth
rate of the population. Note that the birth feedback appears with
some time delay in order to take into account the period of
maturation to become adult. Finally the term $-u$ corresponds to a
normalized death rate, while $\ep>0$ is a scaling parameter.

When $f$ takes the form of the so-called Ricker's function
\begin{equation}\label{f}
f(u)=\slope u e^{-u},\;\; \slope >1,
\end{equation}
equation \eqref{eq} is commonly referred as the {\it diffusive
Nicholson's blowflies equation}. This kind of equation has been
intensively studied in the literature. The purely reactive part,
namely the underlying delay differential equation, has attracted
the attention of many researchers during the past decades (see for
instance \cite{Ruan-2006} and references cited therein). On the
other hand, the diffusive equation has also been extensively
investigated from the spatial propagation point of view, that is
speed of spread, travelling wave solutions (we refer for instance
to So and Zou \cite{So-Zou}, So, Wu and Zou \cite{So-Wu-Zou},
Thieme and Zhao \cite{Thieme-Zhao}, Fang and Zhao
\cite{Fang-Zhao10}, and the references therein).

In this work, we consider the monostable equation \eqref{eq} in
the so-called monotonic regime. Precisely we assume that
$f:[0,\infty)\to [0,\infty)$ is a function of the class $C^2$ such
that
\begin{equation}\label{function-f}
\begin{cases}
f(0)=0,\;\; f(1)=1,\;\; f'(0)>1, \;\; f'(1)<1,\\
f'(u)>0, \;\;\forall u\in (0,1),\\
f(u)>u, \;\;\forall u\in (0,1).
\end{cases}
\end{equation}
In particular, $u\equiv 0$ and $u\equiv 1$ solve \eqref{eq}. If we
come back to example \eqref{f}, assuming $\slope\in (1,e)$ implies
that $f$ satisfies \eqref{function-f}, with $\ln \slope$ playing
the role of 1.

Let us observe that, when $\tau=0$, equation \eqref{eq} reduces to
the monostable reaction-diffusion equation
\begin{equation}\label{interface-monostable}
\partial_t u(t,x)=\ep\Delta u(t,x)+\frac{1}{\ep} F\left(u(t,x)\right),
\end{equation}
with $F(u):=f(u)-u$. In view of \eqref{function-f}, the
nonlinearity $F$ exhibits a monostable dynamics, namely
$F(0)=F(1)=0$, $F(u)>0$ for all $u\in (0,1)$, and $F'(0)>0$ while
$F'(1)<0$. Under these assumptions, solutions of
\eqref{interface-monostable} with compactly supported initial data
have been considered first by Freidlin \cite{Freidlin} with
probabilistic tools, then by Evans and Souganidis \cite{Evans}
with Hamilton-Jacobi techniques (we also refer to \cite{BES, BS}
and the references therein). This problem has been recently
revisited using comparison parabolic arguments in \cite{A-Duc}
(including the case of compactly supported initial data), and
\cite{A-Duc2} (for slowly decaying initial data). Roughly
speaking, for compactly supported initial data with convex and
bounded support, as $\ep\to 0$, the solution of
\eqref{interface-monostable} generates a sharp interface at the
very early stages of the dynamics. Then the interface propagates
through the spatial domain, according to a free boundary problem
with constant speed in the normal direction. This speed turns out
to be the minimal speed of propagation of some underlying
travelling wave solutions.

In the delayed case ($\tau >0$) that we consider, we will show
that the above scenario remains valid under the following
assumption on the initial data $\varphi$ arising in \eqref{eq-ID}.

\begin{assumption}\label{ASS-initial} We assume that $\varphi:[-\tau,0]\times \R^N\to [0,1]$ is a uniformly
continuous function satisfying the following.
\begin{itemize}
\item [(i)] There exists $w_0 \in BUC^2(\R^N,\R)$ such that
$$
\Omega_0:=\{x\in\R^N:\;w_0(x)>0\} $$
 is a nonempty smooth bounded
and convex domain, and
\begin{equation}\label{cond-ordre-1}
w_0(x)\leq \varphi(\theta,x), \;\;\forall (\theta,x)\in
[-\tau,0]\times \R^N.
\end{equation}
\item [(ii)] There exists $\delta>0$ such that
\begin{equation}\label{cond-decolage}
\left|\nabla w_0(x).\nu_{\partial\Omega_0}(x)\right|\geq
\delta,\;\;\forall x\in\Gamma _0:=\partial\Omega_0,
\end{equation}
wherein $\nu_{\partial\Omega_0}(x)$ denotes the outward unit
normal vector to $\Omega_0$ at $x\in\Gamma _0$. \item [(iii)]
There exists $v_0\in BUC(\R^N,[0,1))$ such that
\begin{equation}\label{cond-support}
{\rm supp}\; v_0=\overline{\Omega_0}\,,
\end{equation}
and
\begin{equation}\label{cond-ordre-2}
\varphi(\theta,x)\leq v_0(x), \;\;\forall (\theta,x)\in
[-\tau,0]\times \R^N.
\end{equation}
\end{itemize}
\end{assumption}

\begin{rem}\label{UN} The hypothesis $\Vert v_0 \Vert _\infty
<1$ in $(iii)$ shall be used in the construction of upper barriers
in Section \ref{s:upper-barriers}. Nevertheless, when $\Vert v_0
\Vert _\infty =1$, our main result remains valid under the
additional assumption that $f$ satisfies
\begin{equation}\label{hyp-suppl}
f(K_0 u)\leq K_0 f(u),\;\;\forall u\in [0,1],
\end{equation}
for some constant $K_0>1$. See Remark \ref{UN-bis} for details.
\end{rem}

Before stating our main convergence result let us give some
notations. Under assumption \eqref{function-f}, we denote by
$c^*>0$ the minimal speed of the underlying delayed travelling
waves (see Lemma \ref{LE-monostable} for details). In particular,
there is $\left(U^*,c^*\right)\in C^2(\R)\times (0,\infty)$ such
that $U^*$ is nonincreasing and
\begin{equation*}
\begin{cases}
(U^*)''(z)+c^*(U^*)'(z)+f\left(U^*(z+c^*\tau)\right)-U^*(z)=0,\;\;\forall z\in\R,\\
U^*(-\infty)=1\text{ and }U^*(\infty)=0.
\end{cases}
\end{equation*}
Next, for $c>0$, we denote by $\Gamma ^c:=\bigcup _{t\geq 0}
(\{t\} \times \Gamma^c_t)$ the smooth solution of the free
boundary problem (see subsection \ref{ss:distance} for details)
\[
 (P^c)\quad\begin{cases}
 \, V=c
 \quad \text { on } \Gamma ^c _t \vspace{3pt}\\
 \, \Gamma ^c _t\big|_{t=0}=\Gamma_0,
\end{cases}
\]
with $V$ the normal velocity of $\Gamma ^c_t$ in the exterior
direction, and $\Gamma _0$ the initial interface defined in
\eqref{cond-decolage}. Also, we denote by $\Omega ^c_t$ the region
enclosed by the hypersurface $\Gamma ^c_t$.

Here is the main result of the present paper (see subsection
\ref{ss:well-posedness} for the well-posedness of the initial
value problem \eqref{eq}--\eqref{eq-ID}).

\begin{thm}[Convergence to a propagating interface]\label{THEO-conv}
Let the nonlinearity $f$ be as in \eqref{function-f}. Let the
initial data $\varphi$ satisfy Assumption \ref{ASS-initial}. For
each $\ep>0$, let $u^\ep:[-\ep\tau,\infty)\times \R ^N \to \R$ be
the solution of \eqref{eq}--\eqref{eq-ID}. Then the following
convergence results hold.
\begin{itemize}
\item [(i)] For each $c\in \left(0,c^*\right)$ and each $t_0>0$,
we have
\begin{equation*}
\lim_{\ep\to 0^+} \sup_{t\geq t_0}\sup_{\;x\in
\overline{\Omega_t^c}}\left|1-u^\ep(t,x)\right|=0.
\end{equation*}
\item [(ii)] For each $c>c^*$ and each $t_0>0$, we have
\begin{equation*}
\lim_{\ep\to 0^+} \sup_{t\geq t_0}\sup_{\;x\in \R^N\setminus
\Omega_t^c} u^\ep(t,x)=0.
\end{equation*}
\end{itemize}
\end{thm}

A first step towards Theorem \ref{THEO-conv} consists in proving
that, as $\ep \to 0$, the initial value problem
\eqref{eq}--\eqref{eq-ID} generates a sharp interface after a very
small time of order $O\left(\ep|\ln\ep|\right)$. Then, to analyze
the propagation of the interface, we aim at constructing suitable
sub- and super-solutions. This step is strongly related to the
existence of travelling wave solutions. While the upper barriers
are directly constructed by using monostable travelling fronts,
the construction of lower barriers is much more delicate. This
kind of problem has been solved in several situations. In
\cite{HKLM08}, the authors consider a {\it degenerate}
reaction-diffusion equation, and take advantage of the existence
of {\it sharp} travelling fronts to construct sub-solutions. In
\cite{A-Duc}, the standard Fisher-KPP case is considered. The
construction of lower barriers of propagation is performed by
using the existence of non-monotone (and also not everywhere
positive) travelling waves with speeds $c<c^*$. In the non delayed
case, the existence of such a connection easily follows from a
phase plane analysis. In the delayed case we consider, the
existence of similar waves is far from obvious. The key idea of
the present paper is to construct sub-solutions of propagation by
using travelling waves for a modified time delayed
reaction-diffusion equation with a bistable dynamics. We hope that
such a strategy could be used to understand better the classical
non delayed Fisher-KPP case and also to analyze a larger class of
equations.

The organization of the present paper is as follows. In Section
\ref{s:preli}, we recall known facts on the well-posedness of the
initial value problem \eqref{eq}--\eqref{eq-ID}. We also discuss
the links between monostable  travelling waves associated with
$f$, and bistable ones associated with approximations $f_\eta$ of
$f$. This is necessary to develop the key strategy mentioned
above. In Section \ref{s:generation}, we investigate the
generation of a sharp interface in the very early stages of the
dynamics. This is strongly related with the underlying delay
differential equation. Section \ref{s:motion} is concerned with
the study of the propagation of interface from below. We shall
construct accurate lower barriers by using a bistable
approximation. As a result of Sections \ref{s:generation} and
\ref{s:motion}, we shall prove Theorem \ref{THEO-conv} $(i)$.
Section \ref{s:upper-barriers} deals with the construction of
upper barriers to control the propagation from above. This will
imply Theorem \ref{THEO-conv} $(ii)$.

\section{Preliminary}\label{s:preli}

\subsection{Existence and comparison for \eqref{eq}--\eqref{eq-ID}}\label{ss:well-posedness}

We first state the following comparison principle for monotone
delayed reaction-diffusion equations.

\begin{prop}[Comparison principle]\label{LE-comparaison}
Let $\tau>0$, $T>0$ and $g:\R\to \R$ an increasing and continuous
function be given. Let $\left(u,v\right)\in C\left([-\tau,T]\times
\R^N\right)$ be two bounded functions satisfying
\begin{equation*}
\partial_t u,\;\partial _t v,\;\nabla u,\;\nabla v,\;D^2u,\;D^2v\;\in
L^2_{loc}\left((0,T)\times\R^N\right).
\end{equation*}
Assume
\begin{equation}\label{sub-super}
\begin{split}
&\left(\partial_t -\Delta+1\right)u(t,x)-g\left(u(t-\tau,x)\right)\leq 0\\
&\left(\partial_t
-\Delta+1\right)v(t,x)-g\left(v(t-\tau,x)\right)\geq 0,
\end{split}
\end{equation}
for almost every $(t,x)\in (0,T)\times\R^N$, and
\begin{equation}\label{ordre-initial}
u(\theta,x)\leq v(\theta,x)\quad\text{ for all } (\theta,x)\in
[-\tau,0]\times \R^N. \end{equation} Then  $u(t,x)\leq v(t,x)$,
for all $(t,x)\in [-\tau,T]\times \R^N$.
\end{prop}

\begin{proof}
Let us consider the map $w:=u-v\in C\left([-\tau,T]\times
\R^N\right)$. Since $g$ is increasing, it follows from
\eqref{sub-super} and \eqref{ordre-initial} that $w$ satisfies
\begin{equation*}
\left(\partial_t-\Delta+1\right) w(t,x)\leq 0\;\;\; a.e. \text{ in
$\left(0,\min (T,\tau)\right)\times\R^N$}.
\end{equation*}
Since $w(0,\cdot)\leq 0$, the weak comparison principle
\cite[Proposition 52.8]{Souplet} implies $w\leq 0$ in
$\left(0,\min (T,\tau)\right)\times\R^N$. If $T>\tau$, one can
repeat the argument on $\left(\tau,\min
(T,2\tau)\right)\times\R^N$. This proves the proposition.
\end{proof}

We now introduce some notations. Let $X:={\rm
BUC}\left(\R^N,\R\right)$ be the Banach space of bounded and
uniformly continuous functions from $\R^N$ to $\R$, endowed with
the usual supremum norm. Define also the Banach spaces $\mathcal
C:=C\left([-\tau,0],X\right)$ and $\mathcal
C_0:=C\left([-\tau,0],\R\right)$. For convenience, we identify
$\psi \in \mathcal C$ as a function from $ [-\tau,0]\times \R ^N$
into $\R$ defined by $\psi(\theta,x)=\psi(\theta)(x)$. For each
$\alpha<\beta$, we define
\begin{equation*}
[\alpha,\beta]_{\mathcal C}:=\left\{\psi\in \mathcal
C:\;\;\alpha\leq \psi(\theta,x)\leq \beta,\;\forall (\theta,x)\in
[-\tau,0]\times\R^N\right\},
\end{equation*}
and $[\alpha,\beta]_{\mathcal C_0}:=\mathcal C_0\cap
[\alpha,\beta]_{\mathcal C}$. Next, for any continuous function
$w:[-\tau,\infty)\times \R^N\to \R$, we define $w_t \in \mathcal
C$, $t\geq 0$, by
$$
w_t:(\theta,x) \in [-\tau,0]\times \R ^N\mapsto
w_t(\theta,x)=w(t+\theta,x).
$$

The well-posedness of the initial value problem
\eqref{eq}--\eqref{eq-ID} can classically be investigated via the
theory of abstract functional differential equations: since the
initial data $\varphi \in [0,1]_{\mathcal C}$, the initial value
problem \eqref{eq}--\eqref{eq-ID} admits a unique {\it mild}
solution $u^\ep:[0,\infty)\times \R ^N \to [0,1]$, which is
actually classical on $[\ep \tau, \infty)\times \R ^N$. For more
details, we refer the reader to the monograph of Wu \cite{Wu-book}
and the references cited therein.

\subsection{Monostable and bistable delayed travelling
waves}\label{ss:approx}

As explained in the introduction, the construction of lower
barriers is far from obvious when $\tau >0$. A key idea of the
present paper is to derive the {\it monostable propagation of the
interface from below} from the bistable case. To perform this in
Section \ref{s:motion}, let us first define a family of bistable
approximations by extending the monostable nonlinearity $f$ for
negative values of $u$.

\vskip 5pt

 \noindent{\bf Bistable approximations of $f$.} For
$\eta\in (0,1]$, we introduce an increasing and bounded map
$f_\eta:\R\to \R$ of the class $C^2$ such that
\begin{equation}\label{function-f-modif}
\begin{split}
&f_\eta(u)=f(u)\;\; \forall u\in [0,1]\\
&f_\eta(-\eta)=-\eta\;\; \text{ and } {f_\eta}'(-\eta)<1\\
&f_\eta(u)<u\;\; \forall u\in (-\eta,0)\cup (1,\infty)\;\; \text{
and } f_\eta(u)>u\;\; \forall u\in (-\infty,-\eta)\cup (0,1).
\end{split}
\end{equation}
Observe that $f_\eta$ has exactly three fixed points $-\eta<0<1$,
${f_\eta} '(-\eta)<1 $ and ${f_\eta} '(1)=f'(1)<1$. We also
require that the family $\left\{f_\eta\right\}_{\eta\in (0,1]}$ is
ordered in the sense that:
\begin{equation}\label{function-f-order}
\forall \left(\eta,\eta'\right)\in (0,1]^2,\;\;
\eta<\eta'\;\Rightarrow\;f_{\eta'}(u)\leq f_\eta(u)\;\; \forall
u\in \R.
\end{equation}

\vskip 5pt

\noindent{\bf Travelling waves.} We consider the one dimensional
reaction-diffusion equation with time delay
\begin{equation}\label{eq-1D}
\left(\partial_t-\partial_{xx}+1\right)u(t,x)=f_\eta\left(u(t-\tau,x)\right),\;
t>0,\;x\in\R.
\end{equation}
We denote by $u_\eta\equiv
u_\eta(t,x;\psi):[-\tau,\infty)\times\R^N\to [-\eta,1]$ the
solution of \eqref{eq-1D} with the initial condition
\begin{equation}\label{eq-1D-ID}
u_0(\theta)(x)=u(\theta,x)=\psi\in [-\eta,1]_{\mathcal C}.
\end{equation}

Let us notice that the above initial value problem generates a
strongly continuous and increasing semiflow
$\left\{Q_\eta(t)\right\}_{t\geq 0}$ defined by
\begin{equation*}
\left[Q_\eta(t)\psi\right](\theta,x)=\left(u_\eta\right)_t\left(\theta,x;\psi\right),\;
(\theta,x)\in[-\tau,0]\times \R ^N,
\end{equation*}
and acting $[-\eta,1]_{\mathcal C}$ into itself. Also, it follows
from \eqref{function-f-modif} that, for each $t\geq 0$,
 $Q_\eta(t)[0,1]_{\mathcal C}\subset [0,1]_{\mathcal
C}$ and that $Q(t):=\left(Q_\eta(t)\right)|_{[0,1]_{\mathcal C}}$
does not depend upon $\eta$. Note that $Q_\eta$ exhibits a
bistable dynamics while $Q$ is of monostable type.

Let us state some basic facts on travelling waves sustained by
\eqref{eq-1D}.

\begin{lem}[Bistable Travelling waves]\label{LE-bistable} For
$\eta\in(0,1]$ arbitrary, the following holds.
\begin{itemize}
\item [(i)] There exists a unique speed $c_\eta$ such that
\eqref{eq-1D} has a travelling wave solution $(U_\eta,c_\eta)\in
C^2(\R)\times \R$ whose profile $U_\eta$ is nonincreasing, that is
\begin{equation}\label{eq-TW-bi}
\begin{cases}
{U_\eta}''(z)+c_\eta {U_\eta}'(z)+f_\eta\left(U_\eta(z+c_\eta\tau)\right)-U_\eta(z)=0,\;\;\forall z\in\R,\vspace{2pt}\\
U_\eta(-\infty)=1\; \text{ and }\, U_\eta(\infty)=-\eta.
\end{cases}
\end{equation}
\item [(ii)] There exist two constants $\left(\mu,M\right)\in
(0,\infty)^2$ such that
\begin{equation*}
\begin{cases}
\left|1-U_\eta(z)\right|+\left|-\eta-U_\eta(-z)\right|\leq Me^{\mu z},\;\forall z\leq 0,\vspace{2pt}\\
|{U_\eta}'(z)|+|{U_\eta}''(z)|\leq M e^{-\mu |z|},\;\;\forall
z\in\R.
\end{cases}
\end{equation*}
\item[(iii)] There exists some constant $\gamma>0$ such that, for
any $\psi\in [-\eta,1]_{\mathcal C}$ with
\begin{equation}\label{condition}
\liminf_{x\to -\infty}\min_{\theta\in [-\tau,0]}\psi(\theta,x)>0\;
\text{ and }\, \limsup_{x\to +\infty}\max_{\theta\in
[-\tau,0]}\psi(\theta,x)<0,
\end{equation}
one can find $K=K(\psi)>0$ and $\xi=\xi(\psi)\in\R$ such that
\begin{equation*}
\left|u_\eta(t,x;\psi)-U_\eta(x-c_\eta t+\xi)\right|\leq
Ke^{-\gamma t},\;\;\forall (t,x)\in [0,\infty)\times \R.
\end{equation*}
\end{itemize}
\end{lem}

\begin{proof} Part $(i)$ comes from Schaaf \cite[Theorem 3.13]{Sch-delay}
(see also Fang and Zhao \cite[Theorem 6.4]{Fang-Zhao11}). The
behavior of the profile $(ii)$ can be found in Hupkes and Lunel
\cite[Proposition 2.2.5]{Hup-Lun}. Finally the {\it global
asymptotic stability with phase shift of the wave} $(iii)$ is due
to Smith and Zhao \cite[Theorem 3.3]{Smith-Zhao}.
\end{proof}

We recall that $f$ satisfies  \eqref{function-f}. As far as
monostable travelling waves sustained by
\begin{equation}\label{eq-1D-monostable}
\left(\partial_t-\partial_{xx}+1\right)u(t,x)=f\left(u(t-\tau,x)\right),\;
t>0,\;x\in\R,
\end{equation}
are concerned, we quote the following result from Schaaf
\cite[Theorem 2.7]{Sch-delay} (see also \cite{Liang-Zhao}).

\begin{lem}[Monostable travelling waves]\label{LE-monostable}
 There exists $c^* >
0$ such that \eqref{eq-1D-monostable} has a travelling wave
solution $(U_c,c)\in C^2(\R)\times (0,\infty)$ with $0\leq U _c
\leq 1$, if and only if $c\geq c^*$. In addition, when $c\geq c^*$
the waves are nonincreasing.
\end{lem}

In the sequel we denote by $(U^*,c^*)$ the monostable wave with
minimal speed, that is
\begin{equation}\label{eq-TW-mono}
\begin{cases}
(U^*)''(z)+c^*(U^*)'(z)+f\left(U^*(z+c^*\tau)\right)-U^*(z)=0,\;\;\forall z\in\R,\\
U^*(-\infty)=1\; \text{ and }\, U^*(\infty)=0.
\end{cases}
\end{equation}

To conclude this preliminary, we prove the following result on the
convergence of the bistable speeds $c_\eta$.

\begin{lem}[Convergence of speeds]\label{LE-approx-speed} Let $f$ satisfy \eqref{function-f}. Let $\{f_\eta\}_{\eta \in
(0,1]}$ satisfy \eqref{function-f-modif} and
\eqref{function-f-order}.  Then the family
$\left\{c_\eta\right\}_{\eta\in (0,1]}$ is decreasing and
\begin{equation*}
 c_\eta \nearrow c^*,\; \text{ as }\, \eta\searrow 0.
\end{equation*}
\end{lem}

\begin{proof}  Let $\eta\in (0,1]$ be given. Since $0\leq
U^* \leq 1$ and $f_\eta | _{[0,1]}=f$, $U^*(x-c^*t)$ solves
\eqref{eq-1D}. We can select a $\psi \in [-\eta,1]_{\mathcal C}$
such that \eqref{condition} holds together with
$$
\psi(\theta,x)\leq U^*(x-c^*\theta),\; \forall (\theta,x)\in
[-\tau,0]\times \R ^N.
$$
The comparison principle yields $u_\eta(t,x;\psi)\leq
U^*(x-c^*t)$, so that Lemma \ref{LE-bistable} $(iii)$ implies
\begin{equation*}
U_\eta(x-c_\eta t+\xi)-K e^{-\gamma t}\leq U^*(x-c^*t),
\end{equation*}
for some constants $\gamma >0$, $K>0$ and $\xi\in\R$. Choosing
$x=c^*t$, we get $U_\eta((c^*-c_\eta)t+\xi)-Ke^{-\gamma t}\leq
U^*(0)$; if $c^*<c_\eta$ then letting $t\to \infty$, we collect
$1\leq U^*(0)$, a contradiction. Hence, we have $c_\eta\leq c^*$.

Now, let us take $\eta<\eta'$ in $(0,1]$. In view of
\eqref{function-f-order}, the comparison principle implies
$u_{\eta '}(t,x;\psi)\leq u_ \eta(t,x;\psi)$ for any $\psi \in
[-\eta,1]_{\mathcal C}$. Choosing $\psi$ given by
$\psi(\theta,x)=U_\eta(x-c_\eta \theta)$ and using Lemma
\ref{LE-bistable} $(iii)$, we infer that
\begin{equation*}
U_{\eta'}(x-c_{\eta'} t+\xi ')-K' e^{-\gamma' t}\leq
U_\eta(x-c_\eta t),
\end{equation*}
for some given constants $\gamma'>0$, $K'>0$ and $\xi\in\R$.
Choosing $h\in\R$ such that $U_{\eta'}(h)=0$,  $x=c_{\eta '}t-\xi
'+h$, we get $-K'e^{-\gamma ' t}\leq U_\eta((c_{\eta
'}-c_\eta)t-\xi '+h))$; if $c_{\eta'}>c_\eta$ then letting $t\to
\infty$, we collect $0\leq -\eta$, a contradiction. Hence, we have
$c_{\eta'}\leq c_\eta$.

As a result, there is $\hat c \leq c^*$ such that $c_\eta \nearrow
\hat c$, as $\eta\searrow 0$. To conclude let us make the
normalization $U_\eta(0)=1/2$ for each $\eta$. Classically, by the
interior elliptic estimates and Sobolev embedding theorem, we may
assume that, modulo extraction, $U_\eta \to \hat U$ strongly in
$C^{1,\beta}_{loc}(\R)$ and weakly in $W^{2,p}_{loc} (\R)$,
$1<p<\infty$.  Then $(\hat{U}, \hat c)$ satisfies
\eqref{eq-TW-mono} with $c^*$ replaced by $\hat c$. Lemma
\ref{LE-monostable} then enforces $\hat{c}\geq c^*$. The lemma is
proved.
\end{proof}

\section{Lower barriers for small times}\label{s:generation}

The goal of this section is to prove that, after a very short time
as $\ep \to 0$, the solution $u^\ep:[-\ep\tau,\infty)\times
\R^N\to [0,1]$ of \eqref{eq}--\eqref{eq-ID} is very close to 1 in
$\Omega _0$ (roughly speaking). Precisely, the following holds.

\begin{prop}[Generation of interface from below]\label{PROP3}
Let the initial data $\varphi$ satisfy Assumption
\ref{ASS-initial} $(i)-(ii)$. Denote by $d(0,x)$ the smooth
cut-off signed distance function to $\Gamma _0$ as defined in
subsection \ref{ss:distance} (in particular, $d(0,x)<0$ if and
only if $x\in \Omega _0$). Then there exist $\delta _0>0$, $\alpha
_0>0$, $\rho _0>0$ and $\ep_0>0$ such that, for all $\ep\in
(0,\ep_0)$ and all $(\theta,x)\in [-\tau,0]\times \R^N$, the
following holds.
\begin{equation*}
\text{If }\, d(0,x)\leq -\delta _0 \ep|\ln\ep|\; \text{ then }\,
1-\ep^{\rho _0}\leq u^\ep\left(\alpha
_0\ep|\ln\ep|+\ep\tau+\ep\theta,x\right)\leq 1.
\end{equation*}
\end{prop}

The proof shall be given in the end of the present section. The
idea is to construct a sub-solution based upon the delay
differential equation obtained by neglecting diffusion in
\eqref{eq}.

\subsection{A delay differential equation}\label{ss:dde}

Let us consider the delay differential equation
\begin{equation}\label{eq-delay}
\begin{cases}
\displaystyle \frac{dv}{dt}(t)=f\left(v_t(-\tau)\right)-v(t),\;\; t>0,\\
v_0(\cdot)=\phi(\cdot)\in [0,1]_{\mathcal C_0},
\end{cases}
\end{equation}
where $f$ satisfies \eqref{function-f} (recall that $\mathcal
C_0=C\left([-\tau,0],\R\right)$). Because of the aforementioned
reason, we also need to consider, for $\eta \in(0,1]$, the delay
differential equation
\begin{equation}\label{eq-delay-modif}
\begin{cases}
\displaystyle \frac{dv}{dt}(t)=f_\eta\left(v_t(-\tau)\right)-v(t),\;\;t>0,\\
v_0(\cdot)=\phi(\cdot)\in [-\eta,1]_{\mathcal C_0},
\end{cases}
\end{equation}
where $f_\eta$ was defined in \eqref{function-f-modif}. From
standard results for delay differential equation with
quasi-monotone nonlinearity --- see for instance the monograph of
Smith \cite{Smith}--- the following holds.

\begin{lem}[Well-posedness]\label{LE-semiflow}
For each $\phi\in \mathcal C_0$, \eqref{eq-delay-modif} has a
unique global (mild) solution $v_\eta=
v_\eta(\cdot\,;\phi):[-\tau,\infty)\to\R$ and the semiflow
$V_\eta(t)\phi=V_\eta(t;\phi):=(v_\eta)_t(\cdot\,;\phi)$ is
strongly continuous and monotone increasing on $\mathcal C_0$. It
furthermore satisfies the following properties.
\begin{itemize}
\item [(i)] For each $t\geq 0$, $V_\eta(t)[-\eta,1]_{\mathcal
C_0}\subset [-\eta,1]_{\mathcal C_0}$. \item [(ii)] For each
$t\geq 0$, $V_\eta(t)[0,1]_{\mathcal C_0}\subset [0,1]_{\mathcal
C_0}$. The restriction $V(t):=V_\eta(t)|_{[0,1]_{\mathcal C_0}}$
does not depend upon $\eta$ and, for  $\phi\in [0,1]_{\mathcal
C_0}$, the map $t\mapsto V(t)\phi=V(t;\phi)$ is the mild solution
$v_t(\cdot\,;\phi)$ of \eqref{eq-delay}.
\end{itemize}
\end{lem}

\vskip 5pt

\noindent{\bf Dynamics of the DDE.} We start with a lemma on the
global dynamics of \eqref{eq-delay} on $[0,1]_{\mathcal C_0}$.

\begin{lem}[Stability of $1$]\label{LE-Global-stable}
 The following holds.
\begin{itemize}
\item [(i)] For $\phi\in [0,1]_{\mathcal C_0}\setminus\{0\}$, we
have $\lim_{t\to\infty} V(t)\phi=1$ in $\mathcal C_0$. \item
[(ii)] There exist $\delta _1>0$, $M>0$ and $\lambda>0$ such that,
for all $\phi \in \mathcal C _0$,
\begin{equation*}
\|1-\phi\|_{L^\infty(-\tau,0)} \leq \delta _1
\;\;\Rightarrow\;\;\|1-V(t)\phi\|_{L^\infty(-\tau,0)} \leq
Me^{-\lambda t},\;\;\forall t\geq 0.
\end{equation*}
\end{itemize}
\end{lem}

\begin{proof} Let us prove $(i)$, that is the global
stability of the stationary point $\overline{v}=1$. First, we
consider the case where there is $\zeta\in (0,1)$ such that $\phi
(\theta) \geq \zeta$, for all $\theta \in[-\tau,0]$. Since the
semiflow associated with \eqref{eq-delay} is monotone increasing
and since $V(t)[0,1]_{\mathcal C _0} \subset [0,1]_{\mathcal C
_0}$, it is enough to consider the solution with the constant
$\zeta$ as initial data, that is $V(t;\zeta)=v_t(\cdot\,;\zeta)$.
Since $f(\zeta)>\zeta$, the map $t\mapsto v(t;\zeta)$ is
nondecreasing. Hence we get $\lim _{t\to \infty} v(t;\zeta)=1$,
which in turn implies $\Vert V(t)\zeta -1\Vert _\infty =\sup
_{-\tau \leq \theta \leq 0}|v(t+\theta,\zeta)-1| \to 0$, as $t\to
\infty$. This concludes the proof of $(i)$ for this first case.
Let us now consider the general case. Since $\phi\in
[0,1]_{\mathcal C_0}\setminus \{0\}$, there exist $-\tau<a<b<0$
and $\beta>0$ such that
\begin{equation*}
\phi(\theta)\geq \beta \mathbf 1_{[a,b]}(\theta),\;\;\forall
\theta\in [-\tau,0].
\end{equation*}
{}From \eqref{eq-delay}, we obtain that, for all $t\in (0,\tau]$,
\begin{equation*}
\frac{d}{dt}\left(e^t v(t;\phi)\right)\geq e^{t}f(\beta)\mathbf
1_{[\tau+a,\tau+b]}(t)\geq f(\beta)\mathbf 1_{[\tau+a,\tau+b]}(t).
\end{equation*}
Integrating this from 0 to $\tau$ yields $v(\tau;\phi)\geq
f(\beta)(b-a)$. Now, for all $t\in(\tau,2\tau]$, \eqref{eq-delay}
implies $\frac{d}{dt}\left(e^t v(t;\phi)\right)\geq 0$. Hence
\begin{equation*}
v(t;\phi)\geq e^{\tau-t} v(\tau;\phi)\geq
\zeta:=e^{-\tau}f(\beta)(b-a)>0,\;\;\forall t\in [\tau,2\tau],
\end{equation*}
and we are back to the first case. This completes the proof of
$(i)$.

The proof of $(ii)$ is a direct consequence of the exponential
local stability of $\overline{v}$. Indeed, at this point the
characteristic equation associated to \eqref{eq-delay} reads as
\begin{equation*}
\Delta(\lambda):=\lambda+1-f'(1)e^{-\lambda\tau}=0.
\end{equation*}
Since $f'(1)<1$, all roots have strictly negative real parts and
the result follows (see for instance \cite{Thieme90},
\cite{Hale93} and the references therein).
\end{proof}

Next, we shall prove the following important result.

\begin{prop}[Convergence to 1]\label{PROP1} Let $\phi\geq0$ in $\mathcal C_0\setminus\{0\}$ be given. There
exists $\lambda>0$ such that, for all $\alpha>0$ there exists
$\ep_0=\ep_0(\alpha)>0$ such that, for all $\ep\in (0,\ep_0)$,
\begin{equation*}
1-\ep^{\alpha\lambda/2}\leq
V\left(\alpha|\ln\ep|+t;\ep|\ln\ep|\phi\right)(\theta)\leq
1,\;\;\forall (\theta,t)\in [-\tau,0]\times[0,\infty).
\end{equation*}
\end{prop}

\begin{proof} Let $\phi\geq0$ in $\mathcal C_0\setminus\{0\}$ be given. Recalling
that $f'(0)>1$, let $\delta\in (0,1)$ and $\rho>1$ be such that
\begin{equation}\label{rho}
f(u)\geq \rho u,\;\;\forall u\in [0,\delta].
\end{equation}
Applying Lemma \ref{LE-Global-stable} with $\delta$ as initial
data, we have the existence of constants  $M>0$ and $\lambda>0$
such that
\begin{equation}\label{expo-esti}
0\leq 1-V(t;\delta)(\theta)\leq Me^{-\lambda t},\;\;\forall
(\theta,t)\in [-\tau,0]\times [0,\infty).
\end{equation}
Let $\alpha>0$ be given. Consider $\ep _0>0$ small enough so that
$\ep|\ln\ep|\phi\in [0,\delta]_{\mathcal C_0}$ for all
$\ep\in(0,\ep _0)$. Since $\phi\geq 0$ is in $\mathcal
C_0\setminus \{0\}$, there exist $-\tau<a<b<0$ and $\beta>0$ such
that
\begin{equation*}
\ep|\ln \ep|\phi(\theta)\geq \ep |\ln \ep|\beta \mathbf
1_{[a,b]}(\theta),\;\;\forall \theta\in [-\tau,0].
\end{equation*}
Arguing as in the proof of Lemma \ref{LE-Global-stable} and using
\eqref{rho}, we discover that there is $\zeta >0$ such that, for
$\ep
>0$ small enough,
\begin{equation}\label{CI-rho}
v_\ep(t):=v(t;\ep|\ln \ep|\phi)\geq \zeta \ep |\ln \ep|,\;\;
\forall t\in [\tau,2\tau].
\end{equation}
Next, observe that, for all $0< t\leq \tau$,
$$
\frac{d}{dt}\left(e^t v_\ep(t)\right)=e^tf(\ep|\ln \ep|\phi
(t-\tau))\leq e^\tau \ep |\ln \ep| \Vert \phi \Vert _\infty \Vert
f'\Vert _\infty =: C \ep |\ln \ep|.
$$
Integrating this from 0 to $\tau$, we have $ v_\ep(\tau) \leq
e^{-\tau}(\phi(0)+C\tau )\ep|\ln \ep| <\delta$, for $\ep >0$ small
enough.  Therefore we can define
\begin{equation*}
t^\ep:=\sup\left\{t> 2\tau :\;\;v_\ep(s-\tau)\leq
\delta,\;\;\forall s\in [2\tau,t]\right\}.
\end{equation*}
It then follows from the DDE \eqref{eq-delay} and \eqref{rho} that
\begin{equation}\label{super-solution-rho}
{v_\ep}'(t)\geq \rho v_\ep(t-\tau)-v_\ep(t),\;\; \forall
t\in[2\tau,t^\ep].
\end{equation}
Since $\rho>1$, there is  $a>0$ such that $a+1=\rho e^{-a\tau}$.
Then the map $h:t\mapsto A\ep |\ln \ep|e^{at}$, $A:=\zeta
/e^{2a\tau}$ satisfies
\begin{equation}\label{soussol-rho}
h'(t)= \rho h(t-\tau)-h(t),\;\;\forall t\in [2\tau,t^\ep], \text{
and } h(t)\leq \zeta \ep |\ln \ep|,\;\;t\in [\tau,2\tau].
\end{equation}
It follows from \eqref{super-solution-rho}, \eqref{CI-rho} and
\eqref{soussol-rho} that
\begin{equation*}
v_\ep(t)\geq A \ep |\ln\ep| e^{at},\;\; \forall t\in
[2\tau,t^\ep].
\end{equation*}
In view of $v_\ep(t^\ep-\tau)=\delta$, we have
\begin{equation}\label{temps}
t^\ep\leq \tau+\frac{1}{a}\ln \frac{\delta}{A\ep |\ln\ep|}.
\end{equation}
Now since the map $t\mapsto v_\ep(t)$ is increasing, we deduce
from $v_\ep(t^\ep -\tau)=\delta$ that
\begin{equation*}
v_\ep(t^\ep+t+\theta)\geq \delta,\;\; \forall (\theta,t)\in
[-\tau,0]\times[0,\infty).
\end{equation*}
In view of \eqref{temps}, we have $t^\ep \leq \alpha|\ln \ep|$ for
$\ep >0$ small enough so that
\begin{equation*}
v_\ep(\alpha|\ln \ep|+t+\theta)\geq \delta,\;\; \forall
(\theta,t)\in [-\tau,0]\times[0,\infty).
\end{equation*}
Since the semiflow associated with \eqref{eq-delay} is monotone
increasing on $\mathcal C _0$, we thus have
$$
0\leq 1-v_\ep(\alpha|\ln \ep|+t+\theta)\leq 1-V(\alpha|\ln
\ep|+t;\delta)(\theta),
$$
which combined with \eqref{expo-esti} yields, for $\ep >0$ small
enough,
$$
0\leq 1-v_\ep(\alpha|\ln \ep|+t+\theta)\leq Me^{-\lambda(\alpha
|\ln \ep|+t)}\leq M\ep^{\alpha\lambda}\leq \ep^{\alpha\lambda/2}.
$$
This completes the proof of Proposition \ref{PROP1}.
\end{proof}

\vskip 5pt

\noindent{\bf Derivatives of the semiflow.} Let us now provide
some estimates on the derivatives of the semiflow $V_\eta$ with
respect to the state variable. Our first result is a consequence
of the well-known differentiability result of semiflows generated
by delay differential equations (see for instance \cite{Hale93},
see also \cite{Thieme90} for results on abstract semilinear
problems with Hille-Yosida non-densely defined operator).

\begin{lem}[Derivatives]\label{LE-semiflow-diff}
For each $t> 0$, the map $\phi\in \mathcal C_0\mapsto
V_\eta(t;\phi)\in \mathcal C_0$ provided by Lemma
\ref{LE-semiflow} is of the class $C^2$. For each $\phi_0\in
\mathcal C_0$ and each $\phi\in\mathcal C_0$, the map
$t\in[0,\infty)\mapsto
\partial_\phi V_\eta(t;\phi_0)\cdot\phi\in \mathcal C_0$ is the mild solution of the
non-autonomous equation
\begin{equation}\label{eq-C1}
\begin{cases}
\displaystyle \frac{dv}{dt}(t)=L(t,\phi_0)v_t,\;\;t>0,\\
v(\theta)=\phi(\theta),\;\;\theta\in [-\tau,0],
\end{cases}
\end{equation}
wherein, for each $t> 0$, $L(t,\phi_0):\mathcal C_0\to \R$ is
defined by
\begin{equation}\label{def-L}
L(t,\phi_0)\phi
:={f_\eta}'\left(V_\eta(t;\phi_0)(-\tau)\right)\phi(-\tau)-\phi(0).
\end{equation}
Moreover, for each $\phi_0\in \mathcal C_0$ and each $\phi\in
\mathcal C_0$, the map $t\mapsto
\partial_{\phi,\phi}^2 V_\eta(t;\phi_0)\cdot(\phi,\phi)$ is the solution of
\begin{equation}\label{eq-C2}
\begin{cases}
\displaystyle \frac{dv}{dt}(t)=L(t,\phi_0)v_t+G(t;\phi_0;\phi),\;\;t>0,\\
v(\theta)=0,\;\;\theta\in [-\tau,0],
\end{cases}
\end{equation}
wherein the map $t\mapsto G(t;\phi_0;\phi)$ is defined by
\begin{equation}\label{def-G}
G(t;\phi_0;\phi):={f_\eta}''\left(V_\eta(t;\phi_0)(-\tau)\right)\left[\partial_\phi
V_\eta(t;\phi_0)\cdot \phi(-\tau)\right]^2.
\end{equation}
\end{lem}

Here is an estimate on the first derivative.

\begin{lem}[First derivative]\label{LE-der}
There exist constants $M^+>1$ and $\gamma^+>0$ such that, for all
$\phi_0\in\mathcal C_0$,
\begin{equation*}\label{esti-LE-der}
e^{-\tau} e^{-(t+\theta)}\leq \partial_\phi V_\eta (t;\phi_0)\cdot
1(\theta)\leq M ^+e^{\gamma^+(t+\theta)},\;\;\forall (\theta,t)\in
[-\tau,0]\times [0,\infty).
\end{equation*}
\end{lem}

\begin{proof} Let $\phi_0\in\mathcal C_0$ be given. First,
the semiflow $V_\eta(t)$ being monotone increasing on $\mathcal C
_0$, observe that
\begin{equation}\label{eq-positif}
\partial_\phi V_\eta (t;\phi_0)\cdot 1(\theta)\geq 0,\;\;\forall (\theta,t)\in [-\tau,0]\times
[0,\infty).
\end{equation}
Hence, in view of \eqref{eq-C1} and \eqref{def-L}, the function
$w(t):=\partial_\phi V_\eta (t;\phi_0)\cdot 1(0)$ satisfies
\begin{equation*}
w'(t)\geq -w(t),\;\;\forall t\geq 0,
\end{equation*}
so that $w(t)\geq e^{-t}$, for all $t\geq 0$, which in turn
implies
$$
 \partial_\phi V_\eta (t;\phi_0)\cdot
1(\theta) \geq e^{-(t+\theta)},
$$
for all $(\theta,t)\in [-\tau,0]\times [0,\infty)$ such that
$t+\theta \geq 0$. For the remaining $(\theta,t)\in
[-\tau,0]\times [0,\infty)$ such that $t+\theta <0$, we have
$\partial_\phi V_\eta (t;\phi_0)\cdot 1(\theta)=1\geq e^{-(\tau +
t+\theta)}$. This completes the proof of the left-hand side of the
estimate of the lemma.

Next, choosing a constant $\tilde N>1$ such that
\begin{equation}\label{eq-sup-N}
0\leq {f_\eta}'(u)\leq \tilde N,\;\;\forall u\in\R,
\end{equation}
we infer from \eqref{eq-C1} and \eqref{def-L} that
\begin{equation}\label{supersol}
w'(t)\leq \tilde N w(t-\tau)-w(t),\;\;t>0, \text{ and }
w(\theta)=1,\;\;\theta\in [-\tau,0].
\end{equation}
Observe that the map $h:t\mapsto e^{(\tilde N-1)\tau}e^{(\tilde
N-1)t}$ satisfies
\begin{equation}\label{soussol}
h'(t)\geq \tilde N h(t-\tau)-h(t),\;\;t>0, \text{ and }
h(\theta)\geq 1,\;\;\theta\in [-\tau,0].
\end{equation}
It follows from \eqref{supersol} and \eqref{soussol} that
$w(t)\leq e^{(\tilde N-1)\tau}e^{(\tilde N-1)t}$, for all $t\geq
0$. Arguing as above we get the right-hand side of the estimate of
the lemma.\end{proof}

We pursue with the following estimate on the second derivative.

\begin{lem}[Second derivative]\label{LE-der-der}
There exist constants $K>0$ and $\mu>0$ such that, for all
$\phi_0\in\mathcal C_0$,
\begin{equation*}
\left|\partial_{\phi\phi}V_\eta(t;\phi_0)\cdot(1,1)(\theta)\right|\leq
K e^{\mu (t+\theta)}, \;\;\forall (\theta,t)\in [-\tau,0]\times
[0,\infty).
\end{equation*}
\end{lem}

\begin{proof}  In view of \eqref{def-G} and Lemma
\ref{LE-der}, there exists a constant $A>0$ such that, for all
$\phi_0\in \mathcal C_0$,
\begin{equation*}
|G(t;\phi_0;1)|\leq A e^{2\gamma^+(t-\tau)}, \;\;\forall t \geq 0.
\end{equation*}
Hence, the function
$w(t):=\partial_{\phi\phi}V_\eta(t;\phi_0)\cdot(1,1)(0)$ satisfies
\begin{equation}\label{eq-ineqa}
w'(t) \leq \tilde N w(t-\tau)-w(t)+A
e^{2\gamma^+(t-\tau)},\;\;t>0, \text{ and }
w(\theta)=0,\;\;\theta\in [-\tau,0].
\end{equation}
We look for a super-solution of \eqref{eq-ineqa} in the form
$t\mapsto \tilde K e^{\tilde \mu t}$, for some constants $\tilde
K>0$ and $\tilde \mu>0$ to be determined. This leads us to
\begin{equation}\label{esti1}
\tilde \mu\geq \tilde Ne^{-\tilde \mu\tau}-1+\frac{A}{\tilde K}
e^{-2\gamma ^+\tau +\left(2\gamma^+
-\tilde\mu\right)t},\;\;\forall t
> 0,
\end{equation}
which can be achieved by choosing $\tilde \mu>2\gamma^+$ and
$\tilde K>0$ both large enough. Arguing as in the proof of Lemma
\ref{LE-der}, we end up with constants $K>0$ and $\mu >0$ such
that, for all $\phi_0\in\mathcal C_0$, all $\theta \in[-\tau,0]$,
all $t\geq 0$,
\begin{equation*}
\partial_{\phi\phi}V_\eta\left(t;\phi_0\right)\cdot(1,1)(\theta)\leq K e^{\mu
(t+\theta)}.
\end{equation*}

Next, select $C>0$ such that ${f_\eta}''(u)\geq -C$, for all $u\in
\R$. Then we get $w'(t)\geq -C w(t-\tau)-w(t)-A
e^{2\gamma^+(t-\tau)}$, for which we can construct a sub-solution
$t\mapsto -\tilde K e^{\tilde \mu t}$ as above. This completes the
proof of the lemma.
\end{proof}

As a direct consequence of Lemma \ref{LE-der} and Lemma
\ref{LE-der-der}, we obtain the following estimate.

\begin{prop}[Estimate on derivatives]\label{LE-esti-der}
There exist constants $\widehat K>0$ and $\gamma>0$ such that, for
all $\phi_0\in\mathcal C_0$,
\begin{equation*}
|\partial_{\phi\phi} V_\eta (t;\phi_0)\cdot (1,1)(\theta)|\leq
\widehat{K} e^{\gamma t}
\partial_\phi V_\eta (t;\phi_0)\cdot 1(\theta),
\end{equation*}
for all $(\theta,t)\in [-\tau,0]\times [0,\infty)$.
\end{prop}

\subsection{Construction of lower barriers for small times}\label{ss:construction}

We now provide an accurate lower estimate, for small times, of
$u^\ep:[-\ep\tau,\infty)\times \R^N\to [0,1]$ the solution of
\eqref{eq}--\eqref{eq-ID}.

\begin{prop}[Sub-solutions]\label{PROP2} Let the initial data $\varphi$ satisfy Assumption
\ref{ASS-initial} $(i)$. Then there exist $K>0$, $\alpha>0$ and
$\ep_0>0$ such that, for all $\ep\in \left(0,\ep_0\right)$,
\begin{equation*}
\max\left\{0\,;v_\eta\left(\frac{t}{\ep}; w_0(x)-\ep K\tau
-Kt\right)\right\}\leq u^\ep(t,x),
\end{equation*}
for all $(t,x)\in \left[-\ep\tau,\alpha \ep|\ln\ep|\right]\times
\R ^N$. Here,  $v_\eta= v_\eta(\cdot\,;\phi):[-\tau,\infty)\to\R$
denotes the solution of \eqref{eq-delay-modif} arising in Lemma
\ref{LE-semiflow} and the function $w_0$ is as in
\eqref{cond-ordre-1}.
\end{prop}

\begin{proof}  Let us consider the differential operator
\begin{equation*}
\mathcal L^\ep _\eta [u](t,x):=\partial_t u(t,x)-\ep \Delta
u(t,x)-\frac{1}{\ep}\bigl[f_\eta\left(u(t-\ep\tau,x)\right)-u(t,x)\bigr].
\end{equation*}
Since $f_\eta=f$ on $[0,1]$, we have $\mathcal L^\ep _\eta
\left[u^\ep\right](t,x)\equiv 0$. We look for a sub-solution, at
least for small times, $\underline{u}:[-\ep\tau,\infty)\times
\R^N\to \R$ in the form
\begin{equation*}
\underline{u}(t,x):=v_\eta\left(\frac{t}{\ep};w_0(x)-\ep
K\tau-Kt\right).
\end{equation*}
Straightforward computations yield, for each $t>0$ and each
$x\in\R^N$,
\begin{eqnarray*}
&&\mathcal L^\ep _\eta
\left[\underline{u}\right](t,x)=-V^\ep(t,x)\left[K+\ep \Delta w_0(x)+\ep \frac{W^\ep(t,x)}{V^\ep (t,x)}|\nabla w_0(x)|^2\right]\\
&&+\frac{1}{\ep}\left[\left(\frac{dv_\eta}{dt}+v_\eta\right)\left(\frac{t}{\ep};w_0(x)-\ep
\tau-Kt\right)-f_\eta\left(v_\eta\left(\frac{t}{\ep}-\tau;w_0(x)-Kt\right)\right)\right]
\end{eqnarray*}
where
\begin{equation*}
\begin{split}
&V^\ep(t,x):=\partial_\phi V_\eta\left(\frac{t}{\ep};w_0(x)-\ep K\tau-Kt\right)\cdot 1(0)\,,\\
&W^\ep(t,x):=\partial_{\phi\phi} V_\eta \left(\frac{t}{\ep};
w_0(x)-\ep K\tau-Kt\right)\cdot (1,1)(0)\,.
\end{split}
\end{equation*}
Since the semiflow arising in Lemma \ref{LE-semiflow} is monotone
increasing in $\mathcal C_0$ and since $f_\eta$ is increasing, we
have
\begin{equation*}
\begin{split}
&\left(\frac{d v_\eta}{dt}+v_\eta\right)\left(\frac{t}{\ep};w_0(x)-\ep K\tau-Kt\right)-f_\eta\left(v_\eta\left(\frac{t}{\ep}-\tau;w_0(x)-Kt\right)\right)\\
&\leq
\left(\frac{dv_\eta}{dt}+v_\eta\right)\left(\frac{t}{\ep};w_0(x)-\ep
K\tau-Kt\right)-f_\eta\left(v_\eta\left(\frac{t}{\ep}-\tau;w_0(x)-\ep
K \tau-Kt\right)\right)\\
&=0,
\end{split}
\end{equation*}
since $v_\eta$ solves \eqref{eq-delay-modif}. Hence, using
Proposition \ref{LE-esti-der}, we get, for all $\ep \in(0,1)$,
$t>0$, $x\in\R^N$,
\begin{equation*}
\mathcal L^\ep _\eta \left[\underline{u}\right](t,x)\leq
-V^\ep(t,x)\left[K-\ep \|\Delta w_0\|_\infty-\ep \|\nabla
w_0\|_\infty^2\widehat{K} e^{\gamma \frac{t}{\ep}}\right].
\end{equation*}
Looking at small times, the above implies, for all $\ep \in(0,1)$,
$t\in \left(0,\gamma ^{-1}\ep|\ln\ep|\right)$, $x\in\R^N$,
\begin{equation*}
\mathcal L^\ep _\eta\left[\underline{u}\right](t,x)\leq
-V^\ep(t,x)\left[K-\ep \|\Delta w_0\|_\infty-\|\nabla
w_0\|_\infty^2\widehat{K}\right]\leq 0,
\end{equation*}
if $K>0$ is sufficiently large. Next, concerning initial data, we
have,  for all $\theta\in [-\ep\tau,0]$,
\begin{equation*}
\underline{u}(\theta,x)=w_0(x)-\ep K\tau-K\theta\leq w_0(x)\leq
\varphi\left(\frac{\theta}{\ep},x\right)=u^\ep(\theta,x),
\end{equation*}
where we have used \eqref{cond-ordre-1} and \eqref{eq-ID}. The
comparison principle in Proposition \ref{LE-comparaison} thus
implies that
 $$
 \underline{u}(t,x)\leq u^\ep(t,x),\;\; \forall (t,x) \in
\left[-\ep\tau,\gamma ^{-1}\ep|\ln\ep|\right]\times\R^N. $$
Recalling that $u^\ep\geq 0$, this completes the proof of
Proposition \ref{PROP2}.
\end{proof}

\vskip 5pt

\noindent{\bf Proof of Proposition \ref{PROP3}.} Fix $K>0$ and
$\alpha >0$ as in Proposition \ref{PROP2}. Define $\alpha
_0:=\alpha /2$. For $\phi:\equiv \alpha _0 \in \mathcal C _0
\setminus \{0\}$, let us select $\lambda >0$ as in Proposition
\ref{PROP1} and define $\rho _0:=\alpha _0 \lambda /2$. Also, it
follows from Assumption \ref{ASS-initial} $(ii)$ that there exists
$\delta _0 >0$ such that, for $\ep >0$ small enough,
\begin{equation}
d(0,x)\leq -\delta _0 \ep|\ln \ep| \Longrightarrow w_0(x)\geq 4
\alpha _0 \ep |\ln \ep|.
\end{equation}
Now, for any $-\tau \leq \theta \leq 0$, define $s:=\alpha _0
\ep|\ln \ep|+\ep \tau +\ep \theta$ and take $x$ such that
$d(0,x)\leq -\delta _0 \ep |\ln \ep|$. Since, for $\ep >0$ small
enough, $0\leq s\leq \alpha \ep|\ln \ep|$ and $w_0(x)-\ep K \tau
-K s\geq \alpha _0 \ep|\ln \ep|$, we deduce from Proposition
\ref{PROP2} and Proposition \ref{PROP1} that
$$
u^\ep (s,x)\geq v_\eta (\alpha _0 |\ln \ep|+\tau +\theta; \alpha
_0 \ep |\ln \ep|)\geq 1-\ep ^{\rho _0},
$$
which concludes the proof.\qed

\section{Lower barriers via bistable approximation}\label{s:motion}

 As explained before, our analysis of the
propagation of interface from below is performed by approximating
the monostable function $f$ in a bistable manner (see subsection
\ref{ss:approx}). We start with some preliminaries on smooth
signed distance functions associated with a family of free
boundary problems.

\subsection{Smooth cut-off signed distance functions}\label{ss:distance}

For $c>0$, we denote by $\Gamma ^c:=\bigcup _{t\geq 0} (\{t\}
\times \Gamma^c_t)$ the smooth solution of the free boundary
problem
\[
 (P^c)\quad\begin{cases}
 \, V=c
 \quad \text { on } \Gamma ^c _t \vspace{3pt}\\
 \, \Gamma ^c _t\big|_{t=0}=\Gamma_0,
\end{cases}
\]
where $V$ denotes the normal velocity of $\Gamma ^c_t$ in the
exterior direction. Note that since the region enclosed by $\Gamma
_0$, namely $\Omega _0$, is convex, these solutions do exist for
all $t\geq 0$. Also we can naturally, i.e. in a {\it reversible}
manner, extend these solutions for small negative times by letting
$\Gamma _0$ evolve with speed $-c $. Hence, with a slight abuse of
notation, we consider $\Gamma ^c _t$ for all $t\geq -\ep \tau$,
with $\ep>0$ small enough. For each $t\geq -\ep \tau$, we denote
by $\Omega ^c_t$ the region enclosed by the hypersurface $\Gamma
^c_t$.

Let $\widetilde d$ be the signed distance function to $\Gamma ^c$
defined by
\begin{equation}\label{eq:dist}
\widetilde d (t,x):=
\begin{cases}
-&\hspace{-10pt}\mbox{dist}(x,\Gamma ^c_t)\quad\text{ for }x\in\Omega ^c _t \\
&\hspace{-10pt} \mbox{dist}(x,\Gamma ^c_t) \quad \text{ for }
x\in\R ^N \setminus \Omega ^c _t,
\end{cases}
\end{equation}
where $\mbox{dist}(x,\Gamma ^c_t)$ is the distance from $x$ to the
hypersurface  $\Gamma ^c _t$. We remark that $\widetilde d=0$ on
$\Gamma ^c$ and that $|\nabla \widetilde d|=1$ in a neighborhood
of $\Gamma ^c$.

We now introduce the ``cut-off signed distance function" $d$,
which is defined as follows. Let $T>0$ be given. First, choose
$d_0>0$ small enough so that $\widetilde d$ is smooth in the
tubular neighborhood of $\Gamma ^c$
\[
 \{(t,x) \in [-\ep \tau,T] \times \R ^N:\;\;|\widetilde{d}(t,x)|<3
 d_0\}.
\]
Next let $\zeta(s)$ be a smooth increasing function on $\R$ such
that
\[
 \zeta(s)= \left\{\begin{array}{ll}
 s &\textrm{ if }\ |s| \leq d_0\vspace{4pt}\\
 -2d_0 &\textrm{ if } \ s \leq -2d_0\vspace{4pt}\\
 2d_0 &\textrm{ if } \ s \geq 2d_0.
 \end{array}\right.
\]
We then define the cut-off signed distance function $d$ by
\begin{equation}
d(t,x):=\zeta\left(\tilde{d}(t,x)\right).
\end{equation}

Note that
\begin{equation}\label{norme-un}
\text{ if } \quad |d(t,x)|< d_0 \quad \text{ then }\quad |\nabla
d(t,x)|=1,
\end{equation}
and that the equation of motion $(P^c)$ yields
\begin{equation}\label{interface}
\text{ if } \quad |d(t,x)|< d_0 \quad \text{ then }\quad
 \partial _t d(t,x)+c=0.
\end{equation}
Then the mean value theorem provides a constant $\bar N>0$ such
that
\begin{equation}\label{MVT}
|\partial _t d(t,x)+c|\leq  \bar N|d(t,x)| \quad \textrm{ for all
} (t,x) \in [-\ep \tau,T]\times \R ^N.
\end{equation}
Moreover, there exists a constant $C>0$ such that
\begin{equation}\label{est-dist}
|\nabla d (t,x)|+|\Delta d (t,x)|\leq C\quad \textrm{ for all }
(t,x) \in [-\ep \tau,T]\times \R ^N.
\end{equation}

\subsection{Construction of lower barriers}

Let us recall that $\{f_\eta\}_{\eta\in(0,1]}$ denotes a family of
bistable approximations of $f$ such that \eqref{function-f-modif}
and \eqref{function-f-order} hold. Also, for $\eta\in(0,1]$,
$(U_\eta,c_\eta)$ denotes the travelling wave solution (with time
delay) associated with this bistable $f_\eta$ (see Lemma
\ref{LE-bistable}), namely
\begin{equation}\label{eq-phi}
\left\{\begin{array}{ll}
{U_\eta}''(z)+c_\eta {U_\eta}'(z)+f_\eta\left(U_\eta(z+c_\eta\tau)\right)-U_\eta(z)=0,\;\;\forall z\in\R,\vspace{2pt}\\
U_\eta(-\infty)=1,\quad U_\eta (0)=0,\quad U_\eta(\infty)= -\eta.
\end{array} \right.
\end{equation}

\bigskip

In the spirit of the sub-solutions constructed in \cite{A-Hil-Mat}
for bistable systems, we look for sub-solutions $u_\eta^-$ in the
form
\begin{equation}\label{sub}
u_\eta^-(t,x):=U_\eta \left(\frac{d_\eta (t,x)+\ep |\ln \ep|
p(t)}{\ep}\right)-q(t),
\end{equation}
where
\begin{eqnarray}
p(t)&:=&-\EB+e^{Lt}+ K,\vspace{3pt} \label{def-p}\\
q(t)&:=&\sigma \left(\beta \EB+\ep Le^{Lt}\right).\label{def-q}
\end{eqnarray}
Here, $\sigma$, $\beta$, $L$ and $K$ are positive constants to be
determined, and $d_\eta(t,x)$ denotes the cut-off signed distance
function to the interface starting from $\Gamma _0$ and evolving
with speed $c_\eta$, that is the solution of $(P^{c_\eta})$. As
seen in the previous subsection, this allows to define $u_\eta ^-$
for all $t\geq -\ep \tau$, $x\in \R ^N$.

\begin{prop}[Sub-solutions]\label{lem:sous-sol-motion} One can find positive
constants $\beta$, $\sigma$ and $L$ such that, for all $K>1$, the
function $u^- _\eta$ satisfies, for $\ep >0$ small enough,
$$
\ep\mathcal L ^\ep _\eta [u_\eta ^-](t,x)=\ep \partial _t u_\eta
^-(t,x)-\ep^2 \Delta u_\eta ^- (t,x)-f_\eta\left(u_\eta ^-(t-\ep
\tau,x)\right)+u_\eta ^-(t,x)\leq 0,
$$
for all $t> 0$, $x\in \R^N$.
\end{prop}

\begin{proof} For ease of notation, we drop most of the subscripts $\eta$.
Also we define
\begin{equation}\label{def-argument}
z:=\frac{d(t,x)+\ep |\ln \ep|p(t)}{\ep}.
\end{equation}
We start by evaluating $\ep \mathcal L ^\ep _\eta [u^-](t,x)$. We
compute
\begin{eqnarray*} \ep
\partial _t u^-(t,x)&=&(\partial
_t d(t,x)+\ep|\ln \ep|p'(t))U'(z)-\ep q'(t)\\
\ep^2 \Delta u^- (t,x)&=&|\nabla d|^2(t,x)U''(z)+\ep\Delta d(t,x)
U'(z).\end{eqnarray*} Next, observe that the previous subsection
enables to write
$$
d(t-\ep \tau,x)=d(t,x)+\ep c\tau+\ep \Theta _\ep(t,x),
$$
where the correction $\Theta _\ep$ vanishes close to the interface
and is $\mathcal O (1)$:
\begin{equation}\label{varphi}
\Theta _\ep(t,x)=0\quad\text{ if }|d(t,x)|\leq d_0, \quad \Vert
\Theta _\ep \Vert _{L^\infty}\leq A,
\end{equation}
for some constant $A>0$. Hence, since $p(t)$ increases and $U(z)$
decreases, we have
\begin{eqnarray*}
u^-(t-\ep\tau,x)&=&U\left(\frac{d(t,x)+\ep |\ln \ep|p(t-\ep
\tau)}{\ep}+c\tau+\Theta
_\ep(t,x)\right)-q(t-\ep \tau)\\
&\geq&U\left(\frac{d(t,x)+\ep |\ln \ep|p(t)}{\ep}+c\tau+\Theta
_\ep(t,x)\right)-q(t-\ep \tau).
\end{eqnarray*}
 Since $f$ is increasing we get
\begin{eqnarray*}
f\left(u^-(t-\ep\tau,x)\right)&\geq&
f\Big(U\left(z+c\tau+\Theta_\ep(t,x)\right)-q(t-\ep \tau)\Big)\\
&=&f\Big(U(z+c\tau+\Theta_\ep(t,x))\Big)-q(t-\ep \tau)f'(\theta),
\end{eqnarray*}
for some $U(z+c\tau+\Theta_\ep(t,x))-q(t-\ep \tau)\leq \theta \leq
U(z+c\tau)$. Hence, we have
\begin{eqnarray*}
f\left(u^-(t-\ep\tau,x)\right)&\geq&
f\left(U(z+c\tau)\right)-q(t-\ep \tau)f'(\theta)\\
&&+\Theta _\ep(t,x)(f \circ U)'\left(z+c\tau+\omega
\Theta_\ep(t,x)\right),
\end{eqnarray*}
for some $0\leq\omega \leq 1$. Combining the above estimates with
$U''(z)+cU'(z)+f(U(z+c\tau))-U(z)=0$, we obtain $\ep \mathcal L
^\ep _\eta [u^-] (t,x)\leq E_1+E_2+E_3$ where
\begin{eqnarray*}
E_1&:=& \ep|\ln \ep|p'(t)U'(z)+q(t-\ep \tau)f'(\theta)-q(t)-\ep q'(t)\\
E_2&:=&(\partial _td(t,x)+c-\ep \Delta d(t,x))
U'(z)+\left(1-|\nabla d(t,x)|^2\right) U''(z)\\
E_3&:=&-\Theta _\ep (t,x)(f \circ U)'\left(z+c\tau+\omega
\Theta_\ep(t,x)\right).
\end{eqnarray*}

\medskip

Let us now analyze further the term $E_1$. By using the
expressions \eqref{def-p}, \eqref{def-q} for $p$ and $q$ we obtain
\begin{eqnarray*}
E_1&=&\beta \EB\left(|\ln \ep|U'(z)+\sigma(e^{\beta
\tau}f'(\theta)-1+\beta)\right)\\
&&+\ep L e^{Lt}\left(|\ln \ep|U'(z)+\sigma(e^{-\ep
L\tau}f'(\theta)-1-\ep L)\right)\\
&=:& \beta \EB I_1+\ep L e^{Lt} I_2.
\end{eqnarray*}
Since $f'(-\eta)<1$ and $f'(1)<1$, we can fix small $a>0$ and
$\beta >0$ such that
$$
e^{\beta \tau}f'(u)-1+\beta\leq -\beta,\quad \forall u\in
[-\eta-a,-\eta+a]\cup [1-a,1+a].
$$
In view of $U(-\infty)=1$, $U(\infty)=-\eta$ and inequality
$U(z+c\tau+\Theta_\ep(t,x))-q(t-\ep \tau)\leq \theta \leq
U(z+c\tau)$, there exists a large $z_0$ such that $\theta \in
[-\eta-a,-\eta+a]\cup [1-a,1+a]$ as soon as $|z|\geq z_0$ (by
choosing $\sigma$ small enough to control the $-q(t-\ep \tau)$
term) and the above inequality applies for $s=\theta$. It follows
from $U'(z)\leq 0$ that $I_1 \leq -\sigma \beta$ in the region
$\{|z|\geq z_0\}$. In the compact region $\{|z|\leq z_0\}$, we
have $U'(z)\leq -b$ for some $b>0$ so that $I_1\leq -b|\ln \ep|+C$
so that $I_1\leq -\sigma \beta$ also holds true. The same argument
yields $I_2\leq -\sigma \beta$. Hence
$$
E_1\leq -\sigma \beta ^2 \EB-\ep \sigma \beta L e^{Lt}\leq -\ep
\sigma \beta L.
$$

\medskip

We now conclude the proof of $\ep \mathcal L ^\ep _\eta
[u^-](t,x)\leq 0$. Assume first that $(t,x)$ lies in the tubular
neighborhood $\{|d(t,x)|\leq d_0\}$ of $\Gamma _t$. In view of
\eqref{norme-un} and \eqref{interface}, the term $E_2$ reduces to
$-\ep \Delta d(t,x)U'(z)$. In view of \eqref{varphi}, the term
$E_3$ vanishes. As a result,
$$
\ep \mathcal L^\ep _\eta [u^-](t,x)\leq -\ep \sigma \beta L +\ep
\Vert \Delta d\Vert _{L^\infty} \Vert U'\Vert _{L^\infty(\R)}\leq
0,
$$
if $L>0$ is large enough. Next, if $(t,x)$ is such that
$|d(t,x)|\geq d_0$ then we shall use the exponential decay of the
derivatives of $U$ --- see Lemma \ref{LE-bistable} $(ii)$--- to
control $E_2$ and $E_3$. Indeed in this region, the argument $z$
defined in \eqref{def-argument} satisfies $|z|\geq d_0/(2\ep)$.
Hence, combining the exponential decay of $U'$ and $U''$ with
\eqref{MVT} and \eqref{est-dist}, we get a bound $|E_2|\leq C_2
e^{-C_2\frac{d_0}{2\ep}}$, for some $C_2>0$. Also, it follows from
\eqref{varphi} that $$ |z+c\tau+\omega \Theta _\ep(t,x)|\geq
\frac{d_0}{2\ep}-c\tau -\omega A\geq \frac{d_0}{4\ep},
$$
which in turn provides a bound $|E_3|\leq
C_3e^{-C_3\frac{d_0}{4\ep}}$, for some $C_3>0$. As a result we
collect, for a constant $C>0$,
$$
\ep \mathcal L^\ep _\eta [u^-](t,x)\leq -\ep \sigma \beta L
+Ce^{-C\frac{d_0}{4\ep}}\leq 0,
$$
if $\ep>0$ is small enough. This completes the proof of the lemma.
\end{proof}

In order to apply the comparison principle, we need the following
estimate.

\begin{lem}[Ordering initial data]\label{lem:initialiser} One can find $K>1$ such that, for $\ep >0$ small enough,
$$
u_\eta ^-(t,x)\leq u^\ep(t+\alpha _0 \ep|\ln \ep|+\ep
\tau,x),\quad \text{ for all } -\ep \tau \leq t \leq 0,\; x\in
\R^N,
$$
where $\alpha _0\ep|\ln \ep|$ denotes the \lq\lq generation of
interface from below time'' appearing in Proposition \ref{PROP3}.
\end{lem}

\begin{proof} For ease of notation, we drop most of the subscripts $\eta$.
If $(t,x)$ is such that $d(t,x)\geq -\ep |\ln \ep|p(t)$, then the
decrease of the wave $U$ yields $u^-(t,x)\leq 0$, and there is
nothing to prove. Now let us take $(t,x)$, with $-\ep \tau \leq t
\leq 0$ and $d(t,x)\leq -\ep |\ln \ep|p(t)$. From the generation
of interface from below analysis  we know that (see Proposition
\ref{PROP3})
\begin{equation}\label{ca-genere}
d(0,x)\leq -\delta _0 \ep|\ln \ep| \Longrightarrow 1-\ep ^{\rho
_0}\leq u^\ep(\alpha _0 \ep|\ln \ep|+\ep \tau+t,x)\quad\text{ for
} -\ep\tau \leq t \leq 0.
\end{equation}
Writing $d(0,x)=d(t,x)+\mathcal O (t)$ and using the expression
for $p$ in \eqref{def-p}, we get, for $-\ep \tau \leq t \leq 0$,
\begin{eqnarray*}
d(0,x)&\leq& -\ep|\ln \ep|p(t)+  C \ep \tau\\
&\leq & -\ep |\ln \ep|(-e^{\beta \tau}+e^{-\ep L \tau}+K)+C\ep
\tau\\
&\leq& -\delta _0 \ep |\ln \ep|,
\end{eqnarray*}
for $\ep >0$ small enough, if $K$ is chosen sufficiently large. In
view of \eqref{ca-genere} it suffices to show that $u^-(t,x)\leq
1-\ep^{\rho _0}$, which follows from the vertical shift $q$.
Indeed, the expression for $q$ in \eqref{def-q} shows that
$q(t)\geq \sigma \beta$ for $-\ep \tau \leq t \leq 0$, so that
$u^-(t,x)\leq 1 -\sigma \beta \leq 1-\ep ^{\rho _0}$. The lemma is
proved.
\end{proof}

\vskip 5pt

\noindent{\bf Proof of Theorem \ref{THEO-conv} $(i)$.} From
Proposition \ref{lem:sous-sol-motion}, Lemma \ref{lem:initialiser}
and the comparison principle, we infer that
\begin{equation}\label{lower-prop}
u_\eta ^-(t-\alpha _0 \ep|\ln \ep|-\ep \tau,x)\leq u^\ep
(t,x)\quad \text{ for all } t \geq \alpha _0 \ep|\ln \ep| +\ep
\tau,\; x\in \R^N.
\end{equation}
Let us recall that $u_\eta^-$ is defined in \eqref{sub} and that
$U_\eta(-\infty)=1$. Hence, the convergence to 1 in $\Omega _t
^{c^*}$, as expressed in Theorem \ref{THEO-conv} $(i)$, is a
direct consequence of both Lemma \ref{LE-approx-speed} and the
lower estimate \eqref{lower-prop}.\qed

\section{Global in time upper barriers}\label{s:upper-barriers}

The aim of this section is to construct a super-solution in order
to control the propagation of the solution from above. Let
$(U^*,c^*)$ be the monostable travelling wave with the minimal
speed $c^*>0$ (see Lemma \ref{LE-monostable}), namely
\begin{equation*}
\begin{cases}
(U^*)''(z)+c^*(U^*)'(z)+f\left(U^*(z+c^*\tau)\right)-U^*(z)=0,\;\;\forall z\in\R,\\
(U^*)'(z)<0,\;\, \forall z\in \R,\\
 U^*(-\infty)=1\; \text{ and
}\, U^*(\infty)=0.
\end{cases}
\end{equation*}
Then we shall prove the upper estimate on
$u^\ep:[-\ep\tau,\infty)\times \R^N\to [0,1]$ the solution of
\eqref{eq}--\eqref{eq-ID}.

\begin{prop}[Super-solutions]\label{PROP-above} Let the initial data $\varphi$ satisfy Assumption
\ref{ASS-initial}. Denote by $d(0,x)$ the smooth cut-off signed
distance function to $\Gamma _0$ as defined in subsection
\ref{ss:distance} (in particular, $d(0,x)<0$ if and only if $x\in
\Omega _0$). Then there exists $h\in\R$ such that, for all $\ep>0$
small enough,
\begin{equation*}
u^\ep(t,x)\leq
U^*\left(\frac{d(0,x)-c^*t}{\ep}+h\right),\;\;\forall (t,x)\in
[-\ep\tau,\infty)\times\R^N.
\end{equation*}
\end{prop}

\begin{proof} Since the function $v_0$ appearing in Assumption \ref{ASS-initial} $(iii)$ satisfies $\Vert
v_0\Vert _\infty<1$, we can choose $h\in\R$ such that
$\|v_0\|_\infty\leq U^*(c^*\tau+h)$. Up to changing $U^*$ by
$U^*(\cdot+h)$, we can assume $h=0$ so that
\begin{equation}\label{choceTW}
\|v_0\|_\infty\leq U^*(c^*\tau).
\end{equation}
Let $x_0\in \partial \Omega _0=\Gamma _0$ be given and denote by
$n_0$ the outward unit normal  vector to $\Gamma _0$ at $x_0$.
Then consider the map $u^+:[-\ep\tau,\infty)\times \R ^N \to \R$
defined by
\begin{equation*}
u^+(t,x):=U^*\left(\frac{(x-x_0).n_0-c^*t}{\ep}\right).
\end{equation*}
Setting $z=\displaystyle\frac{(x-x_0).n_0-c^*t}{\ep}$, we compute
\begin{equation*}
\begin{split}
\mathcal L^\ep  [u^+](t,x)&:=\partial _t u^+(t,x)-\ep \Delta
u^+(t,x)-\frac 1 \ep f\left(u^+(t-\ep \tau,x)\right)+\frac 1 \ep
u^+(t,x)\\&=
-\frac{c^*}{\ep} \left(U^*\right)'(z)-\frac{1}{\ep} \left(U^*\right)''(z)-\frac{1}{\ep}f\left(U^*(z+c^*\tau)\right)+\frac 1 \ep U^*(z)\\
&=0,
\end{split}
\end{equation*}
for all $t>0$, $x\in \R^N$. Let us now prove that
\begin{equation*}
u^\ep(\theta,x)=\varphi\left(\frac \theta \ep,x\right)\leq
U^*\left(\frac{(x-x_0).n_0-c^*\theta}{\ep}\right)=u^+(\theta,x),
\end{equation*}
for all $(\theta,x)\in[-\ep\tau,0]\times \R ^N$. In view od
Assumption \ref{ASS-initial} $(iii)$ and the decrease of $U^*$, it
is sufficient to check that
\begin{equation}\label{etoile}
v_0(x)\leq U^*\left(\frac{(x-x_0).n_0}{\ep}+c^*\tau\right), \;\;
\forall x\in \R ^N.
\end{equation}
When $(x-x_0).n_0\leq 0$, the above inequality follows from
\eqref{choceTW}. When $(x-x_0).n_0>0$, \eqref{cond-support} and
the convexity of $\Omega _0$ implies $v_0(x)=0$ and \eqref{etoile}
is clear. Hence, it follows from the comparison principle that
\begin{equation*}
u^\ep(t,x)\leq
U^*\left(\frac{(x-x_0).n_0-c^*t}{\ep}\right),\;\;\forall (t,x)\in
[-\ep\tau,\infty)\times\R^N,
\end{equation*}
for each $x_0\in\partial\Omega _0$. This completes the proof of
the proposition.
\end{proof}

\begin{rem}\label{UN-bis} If $\|v_0\|_\infty=1$ then, under
assumption \eqref{hyp-suppl} of Remark \ref{UN}, we have $\mathcal
L^\ep [K_0 u^+](t,x)\geq 0$. Also, normalizing the travelling wave
$U^*$ by $1=K_0 U^*(c^*\tau)$ and arguing as above, we see that
$u^\ep(\theta,x)\leq K_0 u^+(\theta,x)$, for all
$(\theta,x)\in[-\ep\tau,0]\times\R ^N$. Hence, the comparison
principle yields
\begin{equation*} u^\ep(t,x)\leq K_0
U^*\left(\frac{(x-x_0).n_0-c^*t}{\ep}\right),\;\;\forall (t,x)\in
[-\ep\tau,\infty)\times\R^N,
\end{equation*}
for each $x_0\in\partial\Omega _0$.
\end{rem}

\vskip 5pt

\noindent{\bf Proof of Theorem \ref{THEO-conv} $(ii)$.} The
convergence to 0 outside $\Omega _t ^{c^*}$, as expressed in
Theorem \ref{THEO-conv} $(ii)$, is a direct consequence of the
control from above provided by Proposition \ref{PROP-above}.\qed

\end{document}